\documentclass[11pt]{article}
\usepackage{geometry}
\usepackage{color}
\usepackage{float}
\usepackage{textcomp}
\usepackage{url}
\usepackage{amsthm}
\usepackage{amsmath}
\usepackage{amssymb}%
\usepackage{mdframed}
\usepackage{multirow}
\usepackage{changes}
\usepackage{hyperref}
\usepackage{authblk} 

\usepackage{mathtools}
\usepackage{booktabs}    
\usepackage{subfigure}
\usepackage{caption}
\usepackage{multirow}
\usepackage{multicol}    
\usepackage{array}       
\usepackage{graphicx}
\graphicspath{{figuras/}}
\usepackage{empheq,amsmath}

\usepackage
[ 
lined,
boxruled, 
commentsnumbered
]
{algorithm2e}
\DecMargin{0.1cm} 

\usepackage{cleveref}

\newtheorem{theorem}{Theorem}[section]
\newtheorem{lemma}[theorem]{Lemma}

\newcommand{\bi}{\begin{itemize}}
\newcommand{\ei}{\end{itemize}}
\newcommand{\ba}{\begin{array}}
\newcommand{\ea}{\end{array}}

\newlength{\eqAlgoAfter}
\newlength{\eqAlgoBefore}
\begin{document}

\title{\textbf{A Non-Monotone Line-Search Method for Minimizing Functions with Spurious Local Minima}}


\author[1]{Zohreh Aminifard\thanks{Email: zohreh.aminifard@uclouvain.be. Supported Supported by ``Fonds spéciaux de Recherche", UCLouvain.}}

\author[1]{Geovani Nunes Grapiglia\thanks{Email: geovani.grapiglia@uclouvain.be.}}

\affil[1]{Université catholique de Louvain, Department of Mathematical Engineering, ICTEAM, 1348 Louvain-la-Neuve, Belgium}

\date{February 24, 2025}

\maketitle

\begin{abstract}
In this paper, we propose a new non-monotone line-search method for smooth unconstrained optimization problems with objective functions that have many non-global local minimizers. The method is based on a relaxed Armijo condition that allows a controllable increase in the objective function between consecutive iterations. This property helps the iterates escape from nearby local minimizers in the early iterations. For objective functions with Lipschitz continuous gradients, we derive worst-case complexity estimates on the number of iterations needed for the method to find approximate stationary points. Numerical results are presented, showing that the new method can significantly outperform other non-monotone methods on functions with spurious local minima.
\end{abstract}

\section{Introduction}

In this work, we consider optimization problems of the form
\begin{equation}
\min_{x\in\mathbb{R}^{n}}\,f(x),
\label{eq:zg1}
\end{equation}
where $f:\mathbb{R}^{n}\to\mathbb{R}$ is a continuously differentiable function with many non-global local minimizers, such as the one shown in Figure \ref{fig:0}. The minimization of functions with spurious local minima appears in several applications, including the distance geometry problem \cite{Liberti}, compressive clustering \cite{Ayoub}, and full-waveform inversion \cite{Tristan}. Standard first-order optimization methods, like Gradient Descent and quasi-Newton methods with line-search, impose a decrease in the objective function at consecutive iterations. Therefore, their iterates are prone to getting trapped in the basin of attraction of nearby local minimizers. A common approach to addressing this issue is the multi-start technique \cite{Yuan}, where the method is applied to a set of different starting points, and the best point found is returned as an approximate solution. However, multi-start can be computationally expensive in terms of function and gradient evaluations.
\newpage
\begin{figure}[htp!]
\centering
\includegraphics[width=0.6\columnwidth]{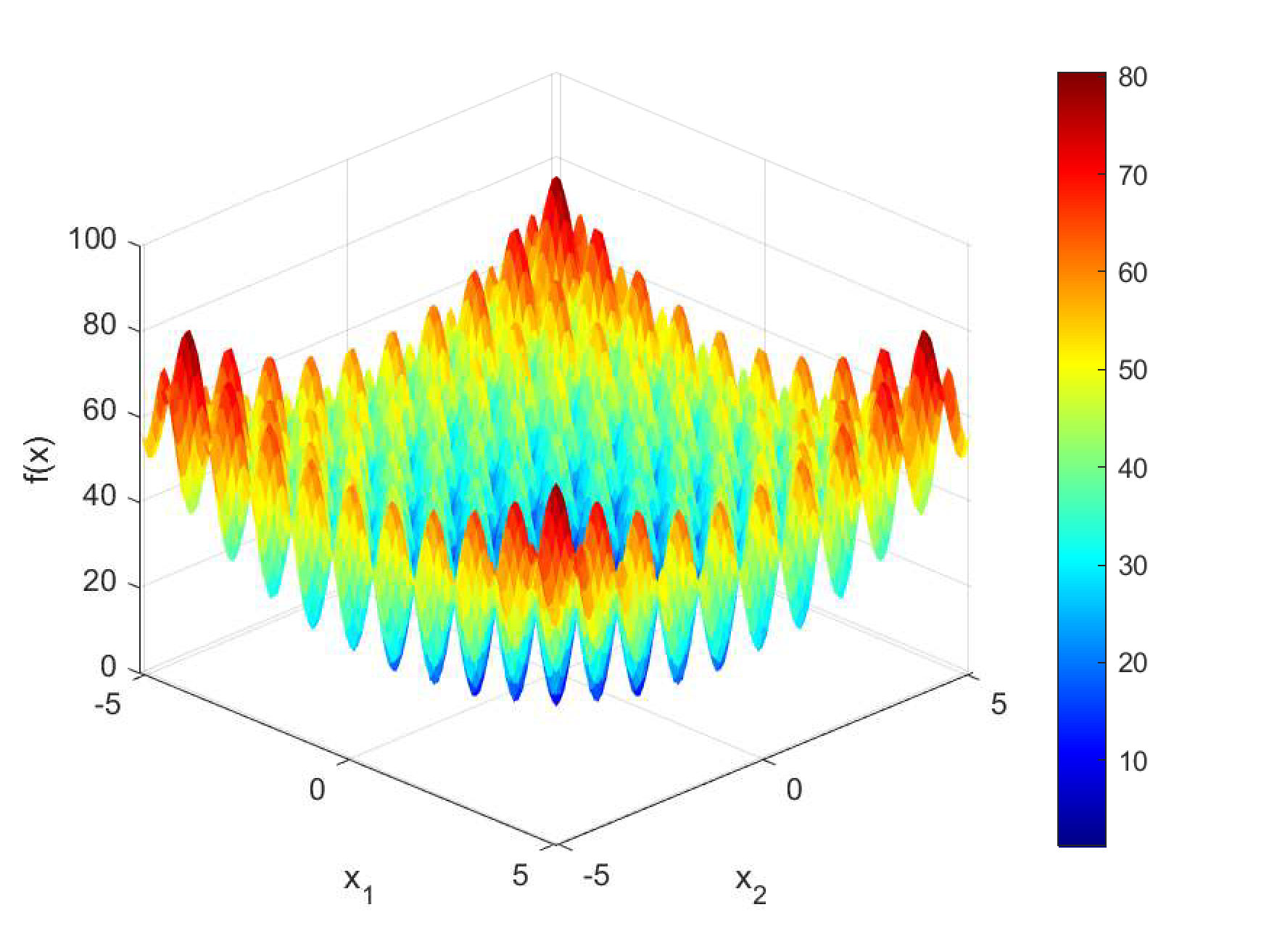}
\caption{The Rastrigin function, defined by $f(x)=20+\sum_{i=1}^{2} \left[x_{i}^{2}-10\cos\left(2\pi x_{i}\right)\right]$.}
\label{fig:0}
\end{figure}

An alternative is to use non-monotone optimization methods \cite{Grippo,ZhangHager,Toint,Cartis,Masoud,Aminifard}, which allow the objective function to increase between consecutive iterations. Due to this property, the iterates may ``jump hills'' and escape nearby local minimizers. Typical non-monotone line-search methods are based on a relaxed variant of the Armijo condition \cite{Sachs}, i.e., given $x_{k}\in\mathbb{R}^{n}$, a search direction $d_{k}$ with $\langle\nabla f(x_{k}),d_{k}\rangle<0$, and a constant $\rho\in (0,1)$, a trial point $x_{k}^{+}=x_{k}+\alpha_{k}d_{k}$ is accepted as the new iterate $x_{k+1}$ whenever the inequality
\begin{equation}
f(x_{k}^{+})\leq f(x_{k})+\rho\alpha_{k}\langle\nabla f(x_{k}),d_{k}\rangle+\nu_{k},
\label{eq:zg2}
\end{equation}
holds, where $\nu_{k}\geq 0$ is the relaxation term that allows non-monotonicity. The larger $\nu_{k}$ is, the greater the chance that the method will accept $x_{k}^{+}$ even when $f(x_{k}^{+})>f(x_{k})$. Different choices for $\nu_{k}$ lead to different non-monotone methods. For example, the Grippo-Lampariello-Lucidi non-monotone line-search \cite{Grippo} corresponds to the choice
\begin{equation}
\nu_{k}=f_{l(k)}-f(x_{k}),
\label{eq:zg3}
\end{equation}
where $f_{l(k)}$ is the maximum value of the objective function obtained in the previous $M$ iterations:
\begin{equation}
f_{l(k)}=\max_{0\leq j\leq m(k)}\,f(x_{k-j}),
\label{eq:zg4}
\end{equation}
with $m(0)=0$ and $m(k)=\min\left\{m(k-1)+1,M\right\}$. More recently, targeting problems with many non-global local minimizers, Grapiglia and Sachs \cite{GrapigliaSachs} proposed a non-monotone line-search inspired by the Metropolis rule used in Simulated Annealing methods \cite{Kirkpatrick,Delahaye}. Specifically, they considered (\ref{eq:zg2}) with
\begin{equation}
\nu_{k}=\sigma\cdot\text{exp}\left(-\max\left\{\theta,\,f(x_{k}^{+})-f(x_{k})\right\}\ln(k+1)\right),
\label{eq:zg5}
\end{equation}
where $\sigma,\theta>0$ are user-defined constants with $\sigma$ large and $\theta$ small. When $f(x_{k}^{+})<f(x_{k})$, it follows from (\ref{eq:zg5}) that $\nu_{k}=\frac{\sigma}{(k+1)^{\theta}}$. Thus, by (\ref{eq:zg2}), it is very likely that $x_{k}^{+}$ will be accepted as the next iterate in early iterations. On the other hand, even when $f(x_{k}^{+})>f(x_{k})$, we also have $\nu_{k}>0$, meaning that $x_{k}^{+}$ can still be accepted as the next iterate. Since the logarithmic factor in (\ref{eq:zg5}) grows with $k$, the likelihood of accepting a point with worse function value decreases as the iterations progress. Moreover, the smaller $\theta$ is, the greater the chance of accepting a point with function value higher than $f(x_{k})$ in the early iterations. The authors in \cite{GrapigliaSachs} reported numerical results showing that the non-monotone method based on (\ref{eq:zg5}) exhibited remarkable performance when applied to the minimization of the Griewank function, a highly nonconvex function with many spurious local minimizers.

Inspired by the results in \cite{GrapigliaSachs}, in this paper we propose a modified Metropolis-based non-monotone line-search method and present comprehensive numerical results on a collection of global optimization test problems. Specifically, our method accepts $x_{k}^{+}$ as the new iterate whenever (\ref{eq:zg2}) is satisfied with
\begin{equation}
\nu_{k}=\sigma\cdot\text{exp}\left(-\max\left\{\theta,\,\dfrac{f_{l(k)}-f(x_{k}^{+})}{\rho\alpha_{k}\langle\nabla f(x_{k}),d_{k}\rangle}\right\}\ln(k+1)\right),
\label{eq:zg6}
\end{equation}
where $\sigma,\theta>0$ are user-defined constants, and $f_{l(k)}$ is defined in (\ref{eq:zg4}). Our numerical results show that the new method can significantly outperform the non-monotone methods in \cite{Grippo,ZhangHager,GrapigliaSachs} on smooth unconstrained optimization problems with objective functions that have many non-global local minimizers.

The paper is organized as follows. In Section 2, we present the new algorithm, its motivation, and the theoretical guarantees regarding the number of iterations to find $\epsilon$-approximate stationary points. In Section 3, we then describe numerical experiments on global optimization problems, comparing the new method with other non-monotone line-search methods.
\vspace{-0.2cm}
\section{New Method and its Complexity Guarantees}\label{Main}

Our new relaxation term (\ref{eq:zg6}) may seem somewhat mysterious at first glance. Therefore, before describing the new algorithm in detail, let us explain how we arrived at the non-monotone term (\ref{eq:zg6}). First, as mentioned above, using the Metropolis-based term (\ref{eq:zg5}) with \( \sigma \gg 1 \) and a small \( \theta > 0 \), it is very likely that \( x_{k}^{+} \) will be accepted in the early iterations when it provides even a slight decrease in the objective function, i.e., when \( f(x_{k}^{+}) < f(x_{k}) \). To further encourage iterates to escape the basin of attraction of nearby local minimizers, it seems reasonable to relax the acceptance rule, allowing \( x_{k}^{+} \) to be accepted in the early iterations as long as \( f(x_{k}^{+}) < f_{l(k)} \), where \( f_{l(k)} \) is defined in (\ref{eq:zg4}). For that, a natural modification of (\ref{eq:zg5}) would be
\begin{equation}
\nu_{k}=\sigma\cdot\text{exp}\left(-\max\left\{\theta,f(x_{k}^{+})-f_{l(k)}\right\}\ln(k+1)\right).
\label{eq:zg7}
\end{equation}
If we suppose that $f(x)$ represents a physical quantity, then it must have an associated \textit{physical dimension} with a corresponding unit of measurement (such as meters, seconds, etc.). From this perspective, the quantity
\begin{equation}
\max\left\{\theta,f(x_{k}^{+})-f_{l(k)}\right\}
\label{eq:zg8}
\end{equation}
in (\ref{eq:zg7}) is not dimensionally consistent, since the user-defined constant $\theta$ is \textit{dimensionless}, whereas $f(x_{k}^{+})-f_{l(k)}$ has a physical dimension. In view of (\ref{eq:zg2}), a natural way to resolve this inconsistency is to replace (\ref{eq:zg8}) in (\ref{eq:zg7}) with
\begin{equation}
\max\left\{\theta,\frac{f(x_{k}^{+})-f_{l(k)}}{-\rho\langle\nabla f(x_{k}),d_{k}\rangle}\right\}.
\label{eq:zg9}
\end{equation}
Indeed, following \cite[p.~322]{Nikita}, let us denote by $[x]$ the unit of measurement associated with any $\bar{x}\in\mathbb{R}^{n}$ and by $[f]$ the unit of measurement associated with any value $f(\bar{x})$. With this notation, a number $z\in\mathbb{R}$ is said to be dimensionless when $[z]=1$. Since $\left[\nabla f(x)\right]=[f]/[x]$ and $[-\rho]=1$, it follows that
\vspace{-0.1cm}
\begin{equation*}
\left[\dfrac{f(x_{k}^{+})-f_{l(k)}}{-\rho\langle\nabla f(x_{k}),d_{k}\rangle}\right]=\dfrac{[f]}{[-\rho][\nabla f(x)][x]}=\dfrac{[f]}{\frac{[f]}{[x]}[x]}=1.
\end{equation*}
\vspace{-0.1cm}
This implies that the fraction in (\ref{eq:zg9}) is dimensionless and thus directly comparable to $\theta$. Based on this \textit{Dimensional Analysis} \cite[Chapter 2]{Zlokarnik}, we arrive at our new relaxation term by replacing (\ref{eq:zg8}) in (\ref{eq:zg7}) with (\ref{eq:zg9}):
\vspace{-0.1cm}
\begin{equation*}
\nu_{k}=\sigma\cdot\text{exp}\left(-\max\left\{\theta,\,\dfrac{f_{l(k)}-f(x_{k}^{+})}{\rho\alpha_{k}\langle\nabla f(x_{k}),d_{k}\rangle}\right\}\ln(k+1)\right).
\end{equation*}
Since \( \langle\nabla f(x_{k}),d_{k}\rangle < 0 \), if \( f(x_{k}^{+}) < f_{l(k)} \), then $\nu_{k} = \sigma/(k+1)^{\theta}$. Therefore, by choosing \( \sigma \) large, it is very likely that \( x_{k}^{+} \) will satisfy (\ref{eq:zg2}) and be accepted as the next iterate in the early iterations of the method, which aligns with our initial intention for the new acceptance rule. Regarding the choice of $\sigma$, since $\nu_{k}$ is compared with function values in (\ref{eq:zg2}), dimensional consistency requires that $[\nu_{k}]=[f]$. For instance, a feasible choice considered in \cite{GrapigliaSachs} is $\sigma=|f(x_{0})|$.

Below we provide the detailed description of our new non-monotone method.
\begin{mdframed}
\noindent\textbf{Algorithm 1.} Non-monotone line-search with modified Metropolis-based relaxation term
\\[0.12cm]
\noindent\textbf{Step 0.} Given $x_{0}\in\mathbb{R}^{n}$, $\alpha_{0},\sigma,\theta>0$, $\beta,\rho\in (0,1)$, and $M\in\mathbb{N}\setminus\left\{0\right\}$, set $m(0)=0$ and $k:=0$.
\\[0.1cm]
\noindent\textbf{Step 1.} Compute a search direction $d_{k}\in\mathbb{R}^{n}$ with $\langle\nabla f(x_{k}),d_{k}\rangle<0$, and set $i:=0$.
\\[0.1cm]
\noindent\textbf{Step 2.1.} Compute $x_{k,i}^{+}=x_{k}+\beta^{i}\alpha_{k}d_{k}$ and
\begin{equation}
\nu_{k,i}=\sigma\cdot\text{exp}\left(-\max\left\{\theta,\dfrac{f_{l(k)}-f(x_{k,i}^{+})}{\rho\beta^{i}\alpha_{k}\langle\nabla f(x_{k}),d_{k}\rangle}\right\}\ln(k+1)\right).
\label{eq:zg10}
\end{equation}
where $f_{l(k)}=\max_{0\leq j\leq m(k)}\,f(x_{k-j})$.
\\[0.2cm]
\noindent\textbf{Step 2.2.} If
\begin{equation}
  f(x_{k,i}^{+})\leq f(x_{k})+\rho\beta^i\alpha_k\langle\nabla f(x_k),d_k\rangle+\nu_{k,i},
\label{eq:zg11}
\end{equation}
\vspace{-0.15cm}
set $i_k=i$, $\nu_k=\nu_{k,i_k}$ and go to Step 3. Otherwise, set $i:=i+1$ and go back to Step 2.1.
\\[0.1cm]
\noindent\textbf{Step 3.} Define $x_{k+1}=x_{k,i_{k}}^{+}$, $\alpha_{k+1}=\beta^{i_{k}-1}\alpha_{k}$,  $m(k+1)=\min\left\{m(k)+1,M\right\}$, set $k:=k+1$, and go to Step 1.
\end{mdframed}
Let us consider the following assumptions:
\\[0.2cm]
\noindent\textbf{A1.} The objective function $f : \mathbb{R}^{n} \to\mathbb{R}$ is differentiable and its gradient $\nabla f : \mathbb{R}^{n}\to \mathbb{R}^{n}$ is
   Lipschitz continuous with the Lipschitz constant $L > 0$.
\\[0.2cm]
\noindent\textbf{A2.} There exists $f_{low}\in\mathbb{R}$ such that $f (x)\geq f_{low}$ for all $x\in\mathbb{R}^{n}$.
\\[0.2cm]
\noindent\textbf{A3.} There exist constants $c_{1},c_{2}>0$, such that
    \begin{equation*}
    \langle\nabla f(x_k),d_k\rangle\leq -c_1\|\nabla f(x_k)\|^2,\ \mbox{and}\ \|d_k\|\leq c_2\|\nabla f(x_k)\|, \ \forall k\geq0.
  \end{equation*}
The next lemma establishes the rates at which Algorithm 1 drives the norm of the gradient of the objective function to zero, depending on the values of the user-defined parameter $\theta$.
  \begin{lemma}
\label{lem:1}
Suppose that A1-A3 hold. Then Algorithm 1 is well-defined and any sequence $\left\{x_{k}\right\}_{k\geq 0}$ generated by it satisfies
\begin{equation}
\min_{k=0,\ldots,T-1}\,\|\nabla f(x_{k})\|^{2}\leq\left\{\begin{array}{ll} \dfrac{\tilde{L}(f(x_{0})-f_{low})}{T}+\dfrac{\tilde{L}\sigma\theta}{(\theta-1)}\dfrac{1}{T},&\text{if $\theta>1$},\\
\dfrac{\tilde{L}(f(x_{0})-f_{low})}{T}+\tilde{L}\sigma\left(\dfrac{1+\ln(T)}{T}\right),&\text{if $\theta=1$},\\
\dfrac{\tilde{L}(f(x_{0})-f_{low})}{T}+\dfrac{\tilde{L}\sigma}{(1-\theta)}\dfrac{1}{T^{\theta}},&\text{if $\theta\in (0,1)$},
\end{array}
\right.
\label{eq:zg12}
\end{equation}
for all $T\geq 1$, where
\begin{equation}
\tilde{L}=\max\left\{\dfrac{1}{\rho\beta\alpha_{0}c_{1}},\dfrac{L}{2\rho (1-\rho)\beta}\left(\frac{c_{2}}{c_{1}}\right)^{2}\right\}
\label{eq:zg13}
\end{equation}
  \end{lemma}

\begin{proof}
Since $f(\,\cdot\,)$ is continuously differentiable (by A1) and $\langle\nabla f(x_{k}),d_{k}\rangle<0$ (by A3), it follows that the relaxed Armijo condition (\ref{eq:zg11}) is satisfied when $\beta^{i}\alpha_{k}$ is sufficiently small. As $\beta \in (0,1)$, this occurs when $i$ is sufficiently large. Hence, by Step 2 of Algorithm 1, the vector $x_{k+1}$, and the values $\nu_{k}$ and $\alpha_{k+1}$, are well-defined. In addition, by Lemma 2 in \cite{GrapigliaSachs2}, we have
\begin{equation}
\alpha_{k}\geq\min\left\{\alpha_{0},\dfrac{2(1-\rho)c_{1}}{Lc_{2}^{2}}\right\},\quad\forall k\geq 0.
\label{eq:zg14}
\end{equation}
Since $\nu_{k}=\nu_{k,i_{k}}$, it follows from (\ref{eq:zg10}) that
\begin{equation*}
    0<\nu_{k}\leq\sigma\cdot e^{-\theta\ln(k+1)}=\dfrac{\sigma}{(k+1)^{\theta}},\quad\forall k\geq 0.
\end{equation*}
Then, by (\ref{eq:zg11}), A3, $\alpha_{k+1}=\beta^{i_{k}-1}\alpha_{k}$ and (\ref{eq:zg14}), we have
\begin{eqnarray}
\dfrac{\sigma}{(k+1)^{\theta}}+f(x_{k})-f(x_{k+1})&\geq & \rho\beta^{i_{k}}\alpha_{k}\left(-\langle\nabla f(x_{k}),d_{k}\rangle\right)\geq \rho\beta\alpha_{k+1}c_{1}\|\nabla f(x_{k})\|^{2}\nonumber\\
&\geq &\min\left\{\rho\beta\alpha_{0}c_{1},\dfrac{2\rho (1-\rho)\beta}{L}\left(\frac{c_{1}}{c_{2}}\right)^{2}\right\}\|\nabla f(x_{k})\|^{2}\nonumber\\
& = & \dfrac{1}{\tilde{L}}\|\nabla f(x_{k})\|^{2}.
\label{eq:zg15}
\end{eqnarray}
Given $T\geq 1$, summing up inequalities (\ref{eq:zg15}) for $k=0,\ldots,T-1$ and using A2, we obtain
\begin{equation*}
\dfrac{T}{\tilde{L}}\min_{k=0,\ldots,T-1}\|\nabla f(x_{k})\|^{2}\leq \sum_{k=0}^{T-1}\dfrac{1}{\tilde{L}}\|\nabla f(x_{k})\|^{2}\leq f(x_{0})-f(x_{T})+\sum_{k=0}^{T-1}\dfrac{\sigma}{(k+1)^{\theta}}\leq f(x_{0})-f_{low}+\sigma\sum_{k=0}^{T-1}\dfrac{1}{(k+1)^{\theta}},
\end{equation*}
and so
\begin{equation}
\min_{k=0,\ldots,T-1}\,\|\nabla f(x_{k})\|^{2}\leq \dfrac{\tilde{L}(f(x_{0})-f_{low})}{T}+\tilde{L}\sigma\left(\dfrac{\sum_{k=0}^{T-1}\frac{1}{(k+1)^{\theta}}}{T}\right).
\label{eq:zg16}
\end{equation}
Finally, (\ref{eq:zg12}) follows from (\ref{eq:zg16}) by using the upper bounds:
\begin{equation*}
    \sum_{k=0}^{T-1}\dfrac{1}{(k+1)^{\theta}}=\sum_{k=1}^{T}\dfrac{1}{k^{\theta}}\leq\left\{\begin{array}{ll} \frac{\theta}{\theta-1},&\text{if $\theta>1$},\\
    1+\ln(T),&\text{if $\theta=1$},\\
    \frac{T^{1-\theta}}{(1-\theta)},&\text{if $\theta\in (0,1)$}.
    \end{array}
    \right.
\end{equation*}
\end{proof}
Given $\epsilon>0$, let
\begin{equation}
T(\epsilon)=\inf\left\{k\in\mathbb{N}\,:\,\|\nabla f(x_{k})\|\leq\epsilon\right\}.
\label{eq:zg17}
\end{equation}
It follows Lemma \ref{lem:1} that $T(\epsilon)<+\infty$. The next result gives explicit upper bounds for the number of iterations that Algorithm 1 takes to find an $\epsilon$-approximate stationary point of $f(\,\cdot\,)$ when $\theta\neq 1$.
\begin{theorem}
Suppose that A1-A3 hold, and let $\left\{x_{k}\right\}_{k\geq 0}$ be generated by Algorithm 1. Given $\epsilon>0$, we have
\begin{equation}
T(\epsilon)\leq\left\{\begin{array}{ll} \left[\tilde{L}(f(x_{0})-f_{low})+\dfrac{\tilde{L}\sigma\theta}{(\theta-1)}\right]\epsilon^{-2},&\text{if $\theta>1$},\\
\max\left\{2\tilde{L}(f(x_{0})-f_{low})\epsilon^{-2},\left[\dfrac{2\tilde{L}\sigma}{(1-\theta)}\right]^{\frac{1}{\theta}}\epsilon^{-\frac{2}{\theta}}\right\},&\text{if $\theta\in (0,1)$},
\end{array}
\right.
\label{eq:zg18}
\end{equation}
where $\tilde{L}$ is defined in (\ref{eq:zg13}).
\label{thm:1}
\end{theorem}

\begin{proof}
If $T(\epsilon)=0$, then (\ref{eq:zg18}) is clearly true. Suppose that $T(\epsilon)\geq 1$, Then, by (\ref{eq:zg17}) we have
\begin{equation}
\epsilon^{2}<\min_{k=0,\ldots,T(\epsilon)-1}\,\|\nabla f(x_{k})\|^{2},
\label{eq:zg19}
\end{equation}
and so, for $\theta>1$, the upper bound on $T(\epsilon)$ in (\ref{eq:zg18}) follows directly from the corresponding case of (\ref{eq:zg12}) with $T=T(\epsilon)$. Regarding the case $\theta\in (0,1)$, let us assume, for the sake of contradiction, that
\begin{equation*}
    T(\epsilon)>\max\left\{2\tilde{L}(f(x_{0})-f_{low})\epsilon^{-2},\left[\dfrac{2\tilde{L}\sigma}{(1-\theta)}\right]^{\frac{1}{\theta}}\epsilon^{-\frac{2}{\theta}}\right\}.
\end{equation*}
Then, we obtain the inequalities
\begin{equation*}
    \dfrac{\tilde{L}(f(x_{0})-f_{low})}{T(\epsilon)}\leq\dfrac{\epsilon^{2}}{2}\quad\text{and}\quad\dfrac{\tilde{L}\sigma}{(1-\theta)}\dfrac{1}{T(\epsilon)^{\theta}}\leq\dfrac{\epsilon^{2}}{2}.
\end{equation*}
As a result, applying the corresponding case in (\ref{eq:zg12}) with $T=T(\epsilon)$ implies that
\begin{equation*}
\min_{k=0,\ldots,T(\epsilon)-1}\,\|\nabla f(x_{k})\|^{2}\leq \epsilon^{2},
\end{equation*}
which contradicts (\ref{eq:zg19}). This proves the case $\theta\in (0,1)$ in (\ref{eq:zg18}).
\end{proof}

If \( \epsilon \in (0,1) \), Theorem \ref{thm:1} implies that Algorithm 1 requires at most \( \mathcal{O}\left(\epsilon^{-\frac{2}{\theta}}\right) \) iterations to find an \( \epsilon \)-approximate stationary point of \( f(\,\cdot\,) \) when \( \theta \in (0,1) \). This complexity bound improves upon the bound of \( \mathcal{O}\left(\epsilon^{-\frac{2(1+\theta)}{\theta}}\right) \) established in Theorem 6 of \cite{GrapigliaSachs} for the Metropolis-based method with the relaxation term (\ref{eq:zg5}). As mentioned earlier, the smaller the value of \( \theta \), the higher the probability of accepting \( x_{k,i}^{+} \) even when \( f(x_{k,i}^{+}) > f(x_{k}) \). Consequently, the likelihood of escaping non-global local minimizers increases. However, our complexity bound also indicates that a smaller \( \theta \) may lead to a higher number of iterations required to reach an \( \epsilon \)-approximate stationary point. Therefore, our result for \( \theta \in (0,1) \) captures the well-known tradeoff between exploration and exploitation in global optimization.

\section{Numerical Results}\label{Num1}

To access the efficiency of Algorithm 1, we performed numerical experiments on a set of global optimization problems. Specifically, we compared the following MATLAB codes:
\\[0.3cm]
\noindent\textbf{M}: the monotone Armijo line-search method, which uses (\ref{eq:zg2}) with $\nu_{k}=0$.
\\[0.3cm]
\noindent\textbf{NM1}: the Grippo-Lampariello-Lucidi non-monotone method \cite{Grippo}, which uses (\ref{eq:zg2}) with $\nu_{k}$ defined by (\ref{eq:zg3})-(\ref{eq:zg4}) with $M=10$.
\\[0.3cm]
\noindent\textbf{NM2}: the Zhang-Hager non-monotone method \cite{ZhangHager}, which uses (\ref{eq:zg2}) with $\nu_{k}=C_{k}-f(x_{k})$, where $C_{0}=f(x_{0})$ and, for all $k\geq 0$,
\begin{equation*}
    C_{k+1}=\dfrac{\eta_{k}Q_{k}C_{k}+f(x_{k+1})}{Q_{k+1}},\quad Q_{k+1}=\eta_{k}Q_{k}+1,
\end{equation*}
with $Q_{0}=1$ and $\eta_{k}=(0.85)/(k+1)$.
\\[0.3cm]
\noindent\textbf{NM3}: the Metroplis-based non-monotone method \cite{GrapigliaSachs}, which uses (\ref{eq:zg2}) with $\nu_{k}$ defined by (\ref{eq:zg5}) with $\theta=2$ and $\sigma=|f(x_{0})|$.
\\[0.3cm]
\noindent\textbf{NM4}: Algorithm 1, with $\theta=2$ and $\sigma=|f(x_{0})|$.
\\[0.3cm]
\noindent In all implementations, we considered the parameters $\alpha_0=1$ and $\beta=\rho=0.5$. The search directions were defined by $d_{k}=-H_{k}\nabla f(x_k)$,
where matrices $H_{k}$ were updated using a safe-guarded BFGS formula:
\begin{equation}\label{Bk}
  H_{k+1}=\left\{
        \begin{array}{ll}
        \left(I-\dfrac{s_ky_k^T}{s_k^Ty_k}\right)H_k\left(I-\dfrac{y_ks_k^T}{s_k^Ty_k}\right)+\dfrac{s_ks_k^T}{s_k^Ty_k}, &   s_k^Ty_k>0,\\
        \\
        H_{k}, & \hbox{otherwise,}
        \end{array}
        \right.
\end{equation}
with $H_0 = I$, $s_k = x_{k+1}-x_k$, and $y_k = \nabla f(x_{k+1})-\nabla f(x_k)$. The tests were performed with MATLAB R2023a, on a PC with processor 13th Gen Intel(R) Core(TM) i5-1345U, and 32 GB of RAM. We considered 20 differentiable objective functions from \cite{Alietal}, as described in Table \ref{table:1}. These functions are particularly challenging to minimize due to the presence of numerous spurious local minima. For each function, we randomly generated 360 starting points. Each pair of a function and a starting point was treated as a separate problem, resulting in a total of 7,200 test problems.
\begin{table}
    \begin{tabular}{|l c | l c|}
    \hline
        \textbf{Function} & \textbf{Dimension ($n$)} & \textbf{Function} & \textbf{Dimension ($n$)} \\
    \hline
        Bohachevsky 1 & 2 & Neumaier 2 & 4 \\
        Bohachevsky 2 & 2 & Neumaier 3 & 10 \\
        Cosine Mixture & 4 & Price’s Transistor Modelling & 9 \\
        Easom Problem & 2 & Rastrigin & 10 \\
        Epistatic Michalewicz & 10 & Schaffer 1 & 2 \\
        Exponential & 10 & Schaffer 2 & 2\\
        Griewank & 2 & Shekel’s Foxholes & 10 \\
        Levy and Montalvo 1  & 3 & Shubert & 2 \\
        Levy and Montalvo 2 & 10 & Sinusoidal & 10 \\
        Modified Langerman & 10 & Storn’s Tchebychev & 9 \\
    \hline
    \end{tabular}
     \captionsetup{format=plain}
     \caption{Functions from \cite{Alietal} used in the numerical experiments.}
    \label{table:1}
\end{table}
\normalsize
The methods were compared using data profiles \cite{MoreWild}, with a budget of 100 simplex gradients\footnote{For a function with $n$ variables, one simplex gradient corresponds to $n+1$ function evaluations.} per problem and a tolerance of $\tau=10^{-7}$. At a high level, a data profile for a method is a graph that shows the percentage of problems approximately solved with tolerance $\tau$ as a function of the number of function evaluations (measured in terms of simplex gradients). Therefore, a higher curve indicates a better-performing method.
\newpage
In our first experiment, we compared the methods M, NM1, NM2, and NM3. The data profiles in Figure \ref{fig:1} indicate that M and NM2 exhibit very similar performance, while NM1 and NM3 perform significantly better. Among them, the Metropolis-based method NM3 outperforms NM1. This result highlights the advantages of allowing strong non-monotonicity when addressing global optimization problems.

\begin{figure}[h!]
\centering
\includegraphics[width=0.6\columnwidth]{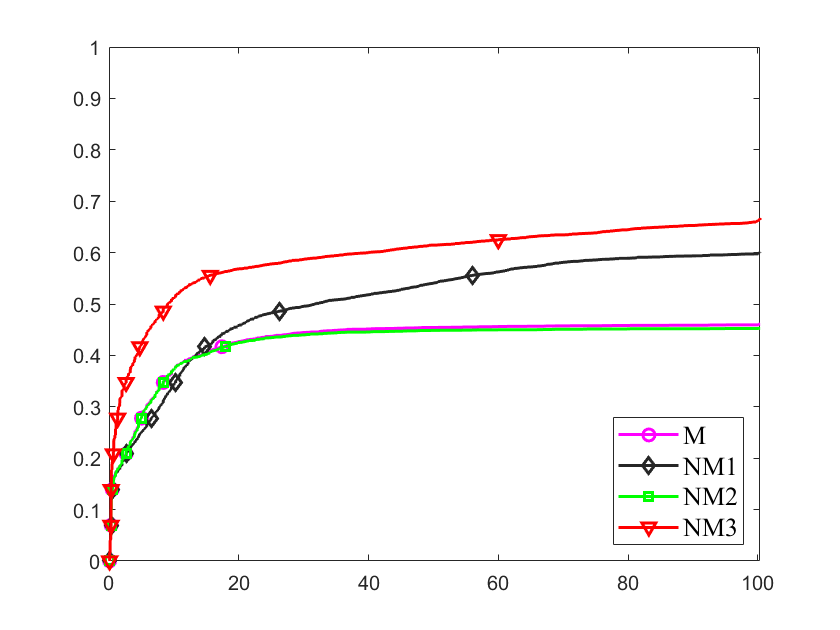}
\caption{Percentage of problems solved as a function of the number of simplex gradients evaluated.}
\label{fig:1}
\end{figure}

In our second experiment, we compared NM3 with NM4, our newly modified Metropolis-based method. Within the given budget of function evaluations, NM4 demonstrated superior performance, as shown in Figure \ref{fig:2}. Specifically, NM4 solved approximately 75\% of the problems, whereas NM3 solved fewer than 60\%.

\begin{figure}[h!]
\centering
\includegraphics[width=0.6\columnwidth]{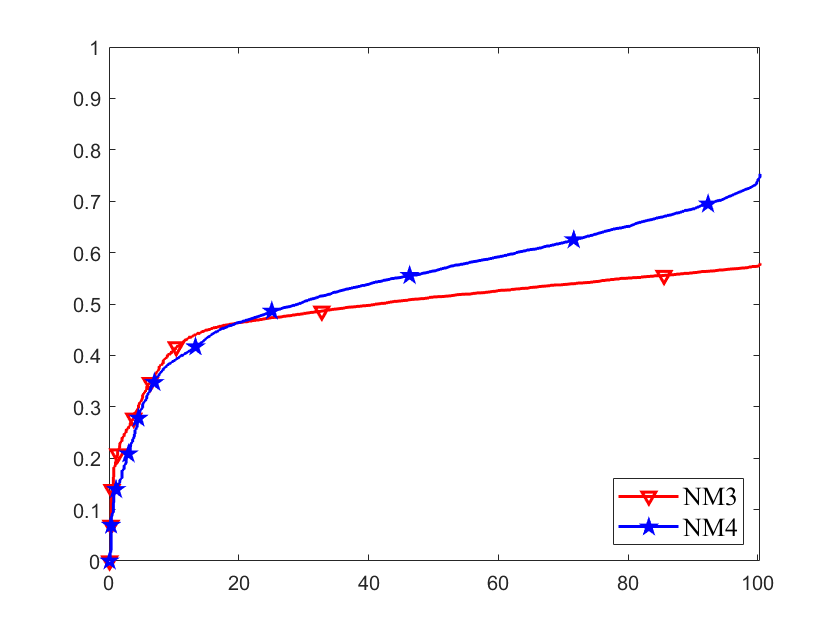}
\caption{Percentage of problems solved as a function of the number of simplex gradients evaluated.}
\label{fig:2}
\end{figure}

\section{Conclusion}\label{Conclusions}

In this paper, we propose a non-monotone line-search method for minimizing functions with numerous spurious local minima. The new method is a relaxed variant of the Metropolis-based non-monotone approach introduced in \cite{GrapigliaSachs}, derived using Dimensional Analysis. The level of non-monotonicity is controlled by a positive parameter $\theta$: the smaller the value of $\theta$, the greater the likelihood of accepting an iterate with a higher function value, potentially allowing the method to explore more promising regions of the domain. Assuming that the objective function is bounded from below and has a Lipschitz continuous gradient, we prove that the proposed method requires at most \(\mathcal{O}(\epsilon^{-2})\) iterations to find an \(\epsilon\)-approximate stationary point when \(\theta > 1\), and $\mathcal{O}\left(\epsilon^{-2/\theta}\right)$ iterations when \(\theta \in (0,1)\). The latter complexity bound highlights the trade-off between exploration and exploitation: choosing \(\theta \ll 1\) encourages greater exploration but may significantly increase the number of iterations required to reach an \(\epsilon\)-approximate stationary point. We also reported numerical experiment comparing monotone and non-monotone line-search methods. The results clearly demonstrate the potential benefits of using non-monotone line search methods for functions with numerous non-global local minimizers. Moreover, the proposed method showed promising performance, surpassing the Metropolis-based approach from \cite{GrapigliaSachs}.


\bibliographystyle{siamplain}

\end{document}